\documentclass[sn-mathphys,Numbered]{sn-jnl}

\usepackage{graphicx}%
\usepackage{multirow}%
\usepackage{amsmath,amssymb,amsfonts}%
\usepackage{amsthm}%
\usepackage{mathrsfs}%
\usepackage[title]{appendix}%
\usepackage{xcolor}%
\usepackage{textcomp}%
\usepackage{manyfoot}%
\usepackage{booktabs}%
\usepackage{algorithm}%
\usepackage{algorithmicx}%
\usepackage{algpseudocode}%
\usepackage{listings}%

\theoremstyle{thmstyleone}%
\newtheorem{theorem}{Theorem}[section]
\newtheorem{proposition}[theorem]{Proposition}%
\newtheorem{corollary}[theorem]{Corollary}%
\theoremstyle{thmstyletwo}%
\newtheorem{example}{Example}%
\newtheorem{remark}[theorem]{Remark}%

\theoremstyle{thmstylethree}%

\begin{document}

\title[Characterizing the $2$-Killing vector fields on multiply twisted product spacetimes]{Characterizing the $2$-Killing vector fields on multiply twisted product spacetimes}


\author[]{\fnm{Adara M.} \sur{Blaga}}\email{adarablaga@yahoo.com}

\affil[]{\orgdiv{Dept. of Mathematics}, \orgname{Faculty of Mathematics and Computer Science, West University of Timi\c soara}, \postcode{300223},
\orgaddress{\city{Timi\c soara}, \country{Romania}}}


\abstract{We characterize the $2$-Killing vector fields on a multiply twisted product manifold, with a special view towards generalized spacetimes.
More precisely, we determine the nonlinear differential equations that completely describe them and the twisted functions, give particular solutions, and construct examples.}




\maketitle

\section{Introduction}

The warped product manifolds, introduced in \cite{Bi} by Bishop and O'Neill, with all their subsequent generalizations, proved to be of interest in physics, especially, in the theory of relativity. The standard spacetime models such as Robertson--Walker, Schwarzschild, static and Kruskal, are all warped products. Also, the simplest models of neighborhoods of stars and black holes are warped products, too \cite{On}. The notion has been further generalized to doubly warped products \cite{erlich}, twisted and doubly twisted products \cite{ponge}, biwarped products \cite{no}, multiply warped products \cite{unal}, sequential warped products \cite{ude}, multiply twisted products \cite{wang}.

Killing vector fields with respect to a pseudo-Riemannian metric are the infinitesimal generators of the isometries. They have already been studied on different types of warped products, see for eg., \cite{wang}. As a generalization, the notion of $2$-Killing vector field has been introduced by Oprea in \cite{op}, and recently, studied on different warped product spaces: in \cite{sh} on warped products, in \cite{ac1} on doubly warped products, and in \cite{ac2} on multiply warped products.

Spacetimes, standing in physics for any four dimensional pseudo-Riemannian manifold, was initially defined as a mathematical model that includes three dimensions of space and one dimension of time. By introducing an extra dimension to the four dimensional spacetime, the five dimensional spacetimes constituted the basement for the development of Kaluza--Klein theory, a field theory that unifies the gravitation and electromagnetism, which have been later generalized to spaces with an arbitrary number of dimensions. For eg., we refer to \cite{car,oh}. One of such generalizations is the Robertson--Walker spacetime, defined as a warped product of an open and connected real interval with the metric $-dt^2$ and an arbitrary Riemannian manifold. This notion has been further extended to the generalized Robertson--Walker spacetimes introduced in \cite{aa}, by allowing for spatial non-homogeneity.
The latest include the Friedmann cosmological models, the Einstein--de Sitter spacetime, the static Einstein spacetime and the de Sitter spacetime. For more results on generalized Robertson--Walker spacetimes, see \cite{mantica}. The warped twisted generalized Robertson--Walker spacetimes extend the biwarped generalized Robertson--Walker spacetimes \cite{misa}, being products of an open and connected real interval $I$ with the metric $-dt^2$ and two arbitrary Riemannian manifolds $(M_i,g_i)$, $i=1,2$, with some twisted functions $f_i$ defined on $I\times M_i$, $i=1,2$.

The aim of the present study is to characterize Killing and $2$-Killing vector fields on multiply twisted product manifolds, with a special view towards spacetimes. More precisely, we determine the nonlinear differential equations that completely describe these vector fields and the twisted functions, give particular solutions to them, and provide examples.

\section{Multiply twisted product manifolds}

Let $(M_i,g_i)$, $i=\overline{1,n}$, be pseudo-Riemannian manifolds, and let $f_i$ for $i=\overline{2,n}$ be positive smooth functions on $M_1\times M_i$.
Then, the \textit{multiply twisted product manifold}
$$M_1\!\times_{{f_2}}\!\!M_2\times \cdots\times_{f_n}\!\!M_n=:(\tilde M,\tilde g)$$ is defined
\cite{wang} as:
$$\tilde M:=M_1\times M_2\times \cdots \times M_n, \ \ \tilde g:=\pi_1^*(g_1)+\sum_{i=2}^n\left(\pi_{1,i}^*(f_i)\right)^2\pi_i^*(g_i)$$
for $\pi_i:M_1\times M_2\times \cdots \times M_n\rightarrow M_i$ the canonical projection on $M_i$, $i=\overline{1,n}$, and $\pi_{1,i}:M_1\times M_2\times \cdots \times M_n\rightarrow M_1\times M_i$ the canonical projection on $M_1\times M_i$, $i=\overline{2,n}$. In this case, the functions $f_i$, $i=\overline{2,n}$, are called the \textit{twisted functions}. The notion of multiply twisted product manifold generalizes the \textit{multiply warped product manifold} introduced in \cite{unal}, where the functions $f_i$, $i=\overline{2,n}$, are positive smooth functions defined only on $M_1$, as well as the \textit{sequential warped product manifold} \cite{ude}, where $n=3$, the function $f_2$ is defined on $M_1$, and the function $f_3$ is defined on $M_1\times M_2$.

For all the rest of the paper, we shall use the same notation for a function on $M_i$, $i=\overline{1,n}$ (or on $M_1\times M_i$, $i=\overline{2,n}$) and its pullback on $\tilde M$, as well as for a metric on $M_i$, $i=\overline{1,n}$, and its pullback on $\tilde M$, and for a vector field on $M_i$, $i=\overline{1,n}$, and its lift on $\tilde M$. Also, we will denote by $\Gamma(TM)$ the set of smooth sections of a smooth manifold $M$.

The tangent bundle of $\tilde M$ orthogonally splits into the direct sum $T\tilde M=\oplus_{i=1}^n TM_i,$
and for any $\tilde X\in \Gamma(T\tilde M)$, we will use the notation
$\tilde X=\sum_{i=1}^nX_i,$ where $X_i\in \Gamma(TM_i)$, $i=\overline{1,n}$.

\bigskip

The first and the second Lie derivatives of the metric $\tilde g$ in the direction of a vector field $V$ on $\tilde M$, defined by:
$$(\pounds_{V}\tilde g)(X,Y):=V(\tilde g(X,Y))-\tilde g([V,X],Y)-\tilde g(X,[V,Y])$$
for any $X,Y\in \Gamma(T\tilde M)$, and, respectively, by
$$\pounds_{V}\pounds_{V}\tilde g:=\pounds_{V}(\pounds_{V}\tilde g)$$
are given as follows.

\begin{proposition}\label{p11a}
Let $(\tilde M,\tilde g)$ be a multiply twisted product manifold and let $V=\sum_{i=1}^nV_i\in \Gamma(T\tilde M)$. Then, for any $X=\sum_{i=1}^nX_i$, $Y=\sum_{i=1}^nY_i\in \Gamma(T\tilde M)$, we have:
\begin{align*}
(\pounds_{V}\tilde g)(X,Y)&=(\pounds_{V_1}g_1)(X_1,Y_1)\\
&\hspace{12pt}+\sum_{i=2}^nf_i^2(\pounds_{V_i}g_i)(X_i,Y_i)+\sum_{i=2}^n\left(V_1(f_i^2)+V_i(f_i^2)\right)g_i(X_i,Y_i),\\
(\pounds_{V}\pounds_{V}\tilde g)(X,Y)&=(\pounds_{V_1}\pounds_{V_1}g_1)(X_1,Y_1)+\sum_{i=2}^nf_i^2(\pounds_{V_i}\pounds_{V_i}g_i)(X_i,Y_i)\\
&\hspace{12pt}+2\sum_{i=2}^n\left(V_1(f_i^2)+V_i(f_i^2)\right)(\pounds_{V_i}g_i)(X_i,Y_i)\\
&\hspace{12pt}+\sum_{i=2}^n\left(V_1(V_1(f_i^2))+V_1(V_i(f_i^2))\right)g_i(X_i,Y_i).
\end{align*}
\end{proposition}
\begin{proof}
We have
\begin{align*}
(\pounds_{V}\tilde g)(X,Y):&=V(\tilde g(X,Y))-\tilde g([V,X],Y)-\tilde g(X,[V,Y])\\
&=V\left(g_1(X_1,Y_1)+\sum_{i=2}^nf_i^2g_i(X_i,Y_i)\right)\\
&\hspace{12pt}-\tilde g\left(\sum_{i=1}^n[V_i,X_i],\sum_{i=1}^nY_i\right)-\tilde g\left(\sum_{i=1}^nX_i,\sum_{i=1}^n[V_i,Y_i]\right)\\
&=V_1(g_1(X_1,Y_1))+\sum_{i=2}^n\left(V_1(f_i^2)+V_i(f_i^2)\right)g_i(X_i,Y_i)\\
&\hspace{12pt}+\sum_{i=2}^nf_i^2V_i(g_i(X_i,Y_i))-g_1([V_1,X_1],Y_1)-\sum_{i=2}^nf_i^2g_i([V_i,X_i],Y_i)\\
&\hspace{12pt}-g_1(X_1,[V_1,Y_1])-\sum_{i=2}^nf_i^2g_i(X_i,[V_i,Y_i])\\
&=(\pounds_{V_1}g_1)(X_1,Y_1)+\sum_{i=2}^nf_i^2(\pounds_{V_i}g_i)(X_i,Y_i)\\
&\hspace{12pt}+\sum_{i=2}^n\left(V_1(f_i^2)+V_i(f_i^2)\right)g_i(X_i,Y_i),
\end{align*}
\begin{align*}
(\pounds_{V}\pounds_{V}\tilde g)(X,Y):&=V((\pounds_{V}\tilde g)(X,Y))-(\pounds_{V}\tilde g)([V,X],Y)-(\pounds_{V}\tilde g)(X,[V,Y])\\
&=V_1\left((\pounds_{V_1}g_1)(X_1,Y_1)\right)+\sum_{i=2}^n\left(V_1(f_i^2)+V_i(f_i^2)\right)(\pounds_{V_i}g_i)(X_i,Y_i)\\
&\hspace{12pt}+\sum_{i=2}^n\left(V_1(V_1(f_i^2))+V_1(V_i(f_i^2))\right)g_i(X_i,Y_i)\\
&\hspace{12pt}+\sum_{i=2}^nf_i^2V_i((\pounds_{V_i}g_i)(X_i,Y_i))\\
&\hspace{12pt}+\sum_{i=2}^n\left(V_1(f_i^2)+V_i(f_i^2)\right)V_i(g_i(X_i,Y_i))\\
&\hspace{12pt}-(\pounds_{V_1}g_1)([V_1,X_1],Y_1)-\sum_{i=2}^nf_i^2(\pounds_{V_i}g_i)([V_i,X_i],Y_i)\\
&\hspace{12pt}-\sum_{i=2}^n\left(V_1(f_i^2)+V_i(f_i^2)\right)g_i([V_i,X_i],Y_i)\\
&\hspace{12pt}-(\pounds_{V_1}g_1)(X_1,[V_1,Y_1])-\sum_{i=2}^nf_i^2(\pounds_{V_i}g_i)(X_i,[V_i,Y_i])\\
&\hspace{12pt}-\sum_{i=2}^n\left(V_1(f_i^2)+V_i(f_i^2)\right)g_i(X_i,[V_i,Y_i])\\
&=(\pounds_{V_1}\pounds_{V_1}g_1)(X_1,Y_1)+\sum_{i=2}^nf_i^2(\pounds_{V_i}\pounds_{V_i}g_i)(X_i,Y_i)\\
&\hspace{12pt}+2\sum_{i=2}^n\left(V_1(f_i^2)+V_i(f_i^2)\right)(\pounds_{V_i}g_i)(X_i,Y_i)\\
&\hspace{12pt}+\sum_{i=2}^n\left(V_1(V_1(f_i^2))+V_1(V_i(f_i^2))\right)g_i(X_i,Y_i),
\end{align*}
and the proof is complete.
\end{proof}

We recall that a smooth vector field $V$ is called a \textit{Killing vector field} with respect to a pseudo-Riemannian metric $g$ if
$\pounds_{V}g=0$, and it is called a \textit{$2$-Killing vector field} \cite{op} if $\pounds_{V}\pounds_{V}g=0.$

From Proposition \ref{p11a}, we can state
\begin{proposition}\label{pd}
If $(\tilde M,\tilde g)$ is a multiply twisted product manifold and $V=\sum_{i=1}^nV_i$ is a smooth vector field on $\tilde M$, then:

(i) $V$ is a Killing vector field if and only if
$$\left\{
                             \begin{array}{ll}
                               \pounds_{V_1}g_1=0\\
                              \pounds_{V_i}g_i=-\frac{\displaystyle V_1(f_i^2)+V_i(f_i^2)}{\displaystyle f_i^2}g_i, \ \ (\forall) \ i=\overline{2,n}
                               \end{array}
                             \right.
;$$

(ii) $V$ is a $2$-Killing vector field if and only if
$$\left\{
                             \begin{array}{ll}
                              \pounds_{V_1} \pounds_{V_1}g_1=0\\
                              \pounds_{V_i}\pounds_{V_i}g_i=-2\frac{\displaystyle V_1(f_i^2)+V_i(f_i^2)}{\displaystyle f_i^2}\pounds_{V_i}g_i-\frac{\displaystyle V_1(V_1(f_i^2))+V_1(V_i(f_i^2))}{\displaystyle f_i^2}g_i, \ \ (\forall) \ i=\overline{2,n}
                               \end{array}
                             \right.
.$$
\end{proposition}

\section{Generalized spacetimes}

We shall consider the multiply twisted product spacetimes of the form
$$I\!\times_{{f_2}}\!\!M_2\times\cdots\times_{f_n}\!\!M_n,$$ where $I$ is an open and connected real interval equipped with the metric $g_1=-dt^2$, and $(M_i,g_i)$ is a Riemannian manifold, $i=\overline{2,n}$.

We recall the following facts.

\begin{proposition}\label{hj}{\rm\cite{ac1}}
Let $(I,g_1)$ and let $V=v\frac{\displaystyle d}{\displaystyle dt}$ be a vector field with $v$ a smooth function on $I$. Then:
\begin{align*}
\pounds_{V}g_1&=2\frac{dv}{dt}g_1;\\
\pounds_{V}\pounds_{V}g_1&=2\Big\{v\frac{d^2v}{dt^2}+2\left(\frac{dv}{dt}\right)^2\Big\}g_1.
\end{align*}

Therefore, $V$ is a Killing vector field on $(I,g_1)$ if and only if $v$ is a constant, and, $V$ is a $2$-Killing vector field on $(I,g_1)$ if and only if
$$v(t)=\sqrt[3]{c_1t+c_2}, \ \ c_1,c_2\in \mathbb{R}.$$
\end{proposition}

\begin{remark}\label{pl}{\rm\cite{ac1}}
Similarly, if we consider $(I,dx^2)$ instead of $(I,-dt^2)$, all the assertions from Proposition \ref{hj} remain true.
\end{remark}

Now we can state
\begin{theorem}\label{t}
Let $I\!\times_{{f_2}}\!\!M_2\times\cdots\times_{f_n}\!\!M_n$ be a multiply twisted product spacetime with the metric $-dt^2+\sum_{i=2}^nf_i^2g_i$. Then:

(i) $V=\sum_{i=1}^nV_i$ is a Killing vector field on $I\!\times_{{f_2}}\!\!M_2\times\cdots\times_{f_n}\!\!M_n$ if and only if
$$\left\{
                             \begin{array}{ll}
                               V_1=c\frac{\displaystyle d}{\displaystyle dt} \  \ c\in \mathbb{R}\\
                              \pounds_{V_i}g_i=-\frac{\displaystyle 1}{\displaystyle f_i^2}\Big\{c\frac{\displaystyle \partial f_i^2}{\displaystyle \partial t}+V_i(f_i^2)\Big\}g_i, \ \ (\forall) \ i=\overline{2,n}
                               \end{array}
                             \right.
;$$

(ii) $V=\sum_{i=1}^nV_i$ is a $2$-Killing vector field on $I\!\times_{{f_2}}\!\!M_2\times \cdots\times_{f_n}\!\!M_n$ if and only if
$$\left\{
                             \begin{array}{ll}
                               V_1=\sqrt[3]{c_1t+c_2}\frac{\displaystyle d}{\displaystyle dt}, \ \ c_1,c_2\in \mathbb{R}\\
                              \pounds_{V_i}\pounds_{V_i}g_i=-\frac{\displaystyle 2}{\displaystyle f_i^2}\Big\{\sqrt[3]{c_1t+c_2}\frac{\displaystyle \partial f_i^2}{\displaystyle \partial t}+V_i(f_i^2)\Big\}\pounds_{V_i}g_i\\
                              \hspace{12pt}-\frac{\displaystyle 1}{\displaystyle f_i^2}\sqrt[3]{c_1t+c_2}\frac{\displaystyle \partial}{\displaystyle \partial t}\Big(\sqrt[3]{c_1t+c_2}\frac{\displaystyle \partial f_i^2}{\displaystyle \partial t}+V_i(f_i^2)\Big)g_i, \ \ (\forall) \ i=\overline{2,n}
                               \end{array}
                             \right..$$
\end{theorem}
\begin{proof}
The two assertions follow from Propositions \ref{pd} and \ref{hj}.
\end{proof}

In particular, if we consider $M_i=I_i$ with $I_i$ an open and connected real interval, denote by $x_i$ the coordinate function on $I_i$ and consider $g_i=dx_i^2$, $i=\overline{2,n}$, then, by means of Remark \ref{pl} and Theorem \ref{t}, we have

\begin{proposition}
A smooth vector field $V=v_1\frac{\displaystyle d}{\displaystyle dt}+\sum_{i=2}^nv_i\frac{\displaystyle d}{\displaystyle dx_i}$ with $v_1$ and $v_i$ smooth functions depending on the coordinate function $t$ on $I$ and $x_i$ on $I_i$ respectively, $i=\overline{2,n}$, on the multiply twisted product spacetime $I\!\times_{{f_2}}\!\!I_2\times \cdots\times_{f_n}\!\!I_n$ with the metric $-dt^2+\sum_{i=2}^nf_i^2dx_i^2$, is

(i) a Killing vector field if and only if
$$\left\{
    \begin{array}{ll}
      v_1=c, \ \ c\in \mathbb{R} \\
      2\frac{\displaystyle dv_i}{\displaystyle dx_i}=-\frac{\displaystyle 1}{\displaystyle f_i^2}\Big\{c\frac{\displaystyle \partial f_i^2}{\displaystyle \partial t}+v_i\frac{\displaystyle \partial f_i^2}{\displaystyle \partial x_i}\Big\}, \ \ (\forall) \ \ i=\overline{2,n}
    \end{array}
  \right.;
$$

(ii) a $2$-Killing vector field if and only if
$$\left\{
    \begin{array}{ll}
      v_1=\sqrt[3]{c_1t+c_2}, \ \ c_1,c_2\in \mathbb{R} \\
      2\Big\{v_i\frac{\displaystyle d^2v_i}{\displaystyle dx_i^2}+2\Big(\frac{\displaystyle dv_i}{\displaystyle dx_i}\Big)^2\Big\}=-\frac{\displaystyle 4}{\displaystyle f_i^2}\Big\{\sqrt[3]{c_1t+c_2}\frac{\displaystyle \partial f_i^2}{\displaystyle \partial t}+v_i\frac{\displaystyle \partial f_i^2}{\displaystyle \partial x_i}\Big\}\frac{\displaystyle dv_i}{\displaystyle dx_i}\\
\hspace{12pt}-\frac{\displaystyle 1}{\displaystyle f_i^2}\sqrt[3]{c_1t+c_2}\Big\{\frac{\displaystyle \partial}{\displaystyle \partial t}\Big(\sqrt[3]{c_1t+c_2}\frac{\displaystyle \partial f_i^2}{\displaystyle \partial t}\Big)+v_i\frac{\displaystyle \partial^2 f_i^2}{\displaystyle \partial t\partial x_i}\Big\}, \ \ (\forall) \ \ i=\overline{2,n}
    \end{array}
  \right..
$$
\end{proposition}

And furthermore
\begin{corollary}\label{cor1}
A smooth vector field $V=\sqrt[3]{c_1t+c_2}\frac{\displaystyle d}{\displaystyle dt}+\sum_{i=2}^n\sqrt[3]{\hat c_ix_i+\check c_i}\frac{\displaystyle d}{\displaystyle d x_i}$ with $c_1, c_2, \hat c_i, \check c_i\in \mathbb{R}$, $i=\overline{2,n}$, on the multiply twisted product spacetime $I\!\times_{{f_2}}\!\!I_2\times \cdots\times_{f_n}\!\!I_n$ with the metric $-dt^2+\sum_{i=2}^nf_i^2dx_i^2$, is

(i) a Killing vector field if and only if
$$\left\{
    \begin{array}{ll}
      c_1=0 \\
     2 \frac{\displaystyle d}{\displaystyle dx_i}(\sqrt[3]{\hat c_ix_i+\check c_i})=-\frac{\displaystyle 1}{\displaystyle f_i^2}\Big\{\sqrt[3]{c_2}\frac{\displaystyle \partial f_i^2}{\displaystyle \partial t}+\sqrt[3]{\hat c_ix_i+\check c_i}\frac{\displaystyle \partial f_i^2}{\displaystyle \partial x_i}\Big\}, \ \ (\forall) \ \ i=\overline{2,n}
    \end{array}
  \right.;
$$

(ii) a $2$-Killing vector field if and only if
\begin{align*}4\Big\{\sqrt[3]{c_1t+c_2}\frac{\displaystyle \partial f_i^2}{\displaystyle \partial t}&+\sqrt[3]{\hat c_ix_i+\check c_i}\frac{\displaystyle \partial f_i^2}{\displaystyle \partial x_i}\Big\}
\frac{\displaystyle d}{\displaystyle dx_i}(\sqrt[3]{\hat c_ix_i+\check c_i})\\
&\hspace{-25pt}+\sqrt[3]{c_1t+c_2}\Big\{\frac{\displaystyle \partial}{\displaystyle \partial t}\Big(\sqrt[3]{c_1t+c_2}\frac{\displaystyle \partial f_i^2}{\displaystyle \partial t}\Big)+\sqrt[3]{\hat c_ix_i+\check c_i}\frac{\displaystyle \partial^2 f_i^2}{\displaystyle \partial t\partial x_i}\Big\}=0, \ \ (\forall) \ \ i=\overline{2,n}
    .
    \end{align*}
\end{corollary}

\section{Particular cases and examples}

Based on the above properties, we shall construct examples of $2$-Killing vector fields on certain multiply twisted product spacetimes.

\begin{proposition}\label{lk}
Let $V=\sqrt[3]{c_1t+c_2}\frac{\displaystyle d}{\displaystyle dt}+\sum_{i=2}^nk_i\frac{\displaystyle d}{\displaystyle dx_i}$ with $c_1, c_2,k_i\in \mathbb{R}$, $i=\overline{2,n}$, be a smooth vector field on the multiply twisted product spacetime $I\!\times_{{f_2}}\!\!I_2\times \cdots\times_{f_n}\!\!I_n$ with the metric $-dt^2+\sum_{i=2}^nf_i^2dx_i^2$. Then:

(i) $V$ is a Killing vector field if and only if
$$\left\{
    \begin{array}{ll}
      c_1=0 \\
      \sqrt[3]{c_2}\frac{\displaystyle \partial f_i^2}{\displaystyle \partial t}+k_i\frac{\displaystyle \partial f_i^2}{\displaystyle \partial x_i}=0, \ \ (\forall) \ \ i=\overline{2,n}
    \end{array}
  \right.;
$$

(ii) if $-\frac{\displaystyle c_2}{\displaystyle c_1}\notin I$, then, $V$ is a $2$-Killing vector field if and only if the twisted functions $f_i:I\times I_i\rightarrow (0,\infty)$ satisfy
$$\frac{\displaystyle \partial^2 f_i^2}{\displaystyle \partial t^2}+\frac{k_i}{\sqrt[3]{c_1t+c_2}}\frac{\displaystyle \partial^2 f_i^2}{\displaystyle \partial t\partial x_i}
      +\frac{\displaystyle c_1}{\displaystyle 3(c_1t+c_2)}\frac{\displaystyle \partial f_i^2}{\displaystyle \partial t}=0, \ \ (\forall) \ \ i=\overline{2,n}
    .$$
\end{proposition}
\begin{proof}
The two assertions follow from Corollary \ref{cor1}.
\end{proof}

\begin{remark}\label{poli2}
Using the notations from Proposition \ref{lk}, if the functions $f_i$ are of the form
$$f_i^2(t,x_i)=e^{a_ix_i+b_i}h(t)$$
with $a_i\in \mathbb{R}^*$, $b_i\in \mathbb{R}$, $i=\overline{2,n}$, and $h$ a strictly monotone positive smooth function on $I$ satisfying
$$\frac{d^2h}{dt^2}+\Big(\frac{a_ik_i}{\sqrt[3]{c_1t+c_2}}+\frac{\displaystyle c_1}{\displaystyle 3(c_1t+c_2)}\Big)\frac{dh}{dt}=0, \ \ (\forall) \ \ i=\overline{2,n}$$
with $c_1\neq 0$ and $-\frac{\displaystyle c_2}{\displaystyle c_1}\notin I$,
then,
$$\frac{\frac{\displaystyle d^2h}{\displaystyle dt^2}}{\frac{\displaystyle dh}{\displaystyle dt}}=-\Big(\frac{a_ik_i}{\sqrt[3]{c_1t+c_2}}+\frac{\displaystyle c_1}{\displaystyle 3(c_1t+c_2)}\Big),$$
which, by integration, gives
$$\ln \Big|\frac{dh}{dt}\Big|=-\Big(\frac{3a_ik_i\sqrt[3]{(c_1t+c_2)^2}}{2c_1}+\frac{1}{3}\ln (|c_1t+c_2|)\Big)+c_0$$
with $c_0\in \mathbb{R}$,
hence,
$$\Big|\frac{dh}{dt}\Big|=\frac{e^{c_0}}{\sqrt[3]{|c_1t+c_2|}e^{\frac{3a_ik_i\sqrt[3]{(c_1t+c_2)^2}}{2c_1}}}.$$

If $k_i\neq 0$, we get
$$f_i(t,x_i)=e^{\frac{a_ix_i+b_i}{2}}\sqrt{\Big|\pm \frac{1}{a_ik_i}e^{c_0-\frac{3a_ik_i\sqrt[3]{(c_1t+c_2)^2}}{2c_1}}+c'_0\Big|}$$
with $c'_0\in \mathbb{R}$ such that $c'_0\pm \frac{\displaystyle 1}{\displaystyle a_ik_i}e^{c_0-\frac{3a_ik_i\sqrt[3]{(c_1t+c_2)^2}}{2c_1}}\neq 0$ for any $t\in I$.
In this case, according to Proposition \ref{lk} (ii), $V=\sqrt[3]{c_1t+c_2}\frac{\displaystyle d}{\displaystyle dt}+\sum_{i=2}^nk_i\frac{\displaystyle d}{\displaystyle dx_i}$ is a $2$-Killing vector field on $I\!\times_{{f_2}}\!\!I_2\times \cdots\times_{f_n}\!\!I_n$.

If $k_i=0$, then
$$f_i(t,x_i)=e^{\frac{a_ix_i+b_i}{2}}\sqrt{\Big|\pm\frac{3e^{c_0}}{2c_1}\sqrt[3]{(c_1t+c_2)^2}+c'_0\Big|}$$
with $c'_0\in \mathbb{R}$ such that $c'_0\pm\frac{\displaystyle 3e^{c_0}}{\displaystyle 2c_1}\sqrt[3]{(c_1t+c_2)^2}\neq 0$ for any $t\in I$.
In this case, according to Proposition \ref{lk} (ii), $V=\sqrt[3]{c_1t+c_2}\frac{\displaystyle d}{\displaystyle dt}$ is a $2$-Killing vector field on $I\!\times_{{f_2}}\!\!I_2\times \cdots\times_{f_n}\!\!I_n$.
\end{remark}

\begin{example}
Let $I\!\times_{{f_2}}\!\!I_2\times \cdots\times_{f_n}\!\!I_n$ be equipped with the metric $-dt^2+\sum_{i=2}^nf_i^2dx_i^2$, where $f_i:I\times I_i\rightarrow (0,\infty)$ are defined by $$f_i(t,x_i)=\sqrt[4]{e^{2x_i-3\sqrt[3]{t^2}}}.$$
Then, according to Remark \ref{poli2},
$$V=\sqrt[3]{t}\frac{\displaystyle d}{\displaystyle dt}+\sum_{i=2}^n\frac{\displaystyle d}{\displaystyle dx_i}$$
is a $2$-Killing vector field on $I\!\times_{{f_2}}\!\!I_2\times \cdots\times_{f_n}\!\!I_n$.
\end{example}

\begin{example}
Let $I\!\times_{{f_2}}\!\!I_2\times \cdots\times_{f_n}\!\!I_n$ be equipped with the metric $-dt^2+\sum_{i=2}^nf_i^2dx_i^2$, where $f_i:I\times I_i\rightarrow (0,\infty)$ are defined by $$f_i(t,x_i)=\sqrt[3]{t}\sqrt{e^{x_i}}.$$
Then, according to Remark \ref{poli2},
$$V=\sqrt[3]{t}\frac{\displaystyle d}{\displaystyle dt}$$
is a $2$-Killing vector field on $I\!\times_{{f_2}}\!\!I_2\times \cdots\times_{f_n}\!\!I_n$.
\end{example}

\begin{proposition}\label{poli5}
Let $V=c\frac{\displaystyle d}{\displaystyle dt}+\sum_{i=2}^nk_i\frac{\displaystyle d}{\displaystyle dx_i}$ with $c,k_i\in \mathbb{R}$, $i=\overline{2,n}$, be a smooth vector field on the multiply twisted product spacetime $I\!\times_{{f_2}}\!\!I_2\times \cdots\times_{f_n}\!\!I_n$ with the metric $-dt^2+\sum_{i=2}^nf_i^2dx_i^2$. Then:

(i) $V$ is a Killing vector field if and only if
$$c\frac{\displaystyle \partial f_i^2}{\displaystyle \partial t}+k_i\frac{\displaystyle \partial f_i^2}{\displaystyle \partial x_i}=0, \ \ (\forall) \ \ i=\overline{2,n};$$

(ii) if $c\neq 0$, then, $V$ is a $2$-Killing vector field if and only if the twisted functions $f_i:I\times I_i\rightarrow (0,\infty)$ satisfy
$$c\frac{\displaystyle \partial^2 f_i^2}{\displaystyle \partial t^2}+k_i\frac{\displaystyle \partial^2 f_i^2}{\displaystyle \partial t\partial x_i}=0, \ \ (\forall) \ \ i=\overline{2,n}.$$
\end{proposition}
\begin{proof}
The two assertions follow from Corollary \ref{cor1}.
\end{proof}

\begin{remark}\label{poli3}
Using the notations from Proposition \ref{poli5}, if the functions $f_i$ are of the form
$$f_i^2(t,x_i)=e^{a_ix_i+b_i}h(t)$$
with $a_i\in \mathbb{R}^*$, $b_i\in \mathbb{R}$, $i=\overline{2,n}$, and $h$ a strictly monotone positive smooth function on $I$ satisfying
$$c\frac{d^2h}{dt^2}+a_ik_i\frac{dh}{dt}=0, \ \ (\forall) \ \ i=\overline{2,n}$$
with $c\neq 0$, then,
$$\frac{\frac{\displaystyle d^2h}{\displaystyle dt^2}}{\frac{\displaystyle dh}{\displaystyle dt}}=-\frac{a_ik_i}{c},$$
which, by integration, gives
$$\ln \Big|\frac{dh}{dt}\Big|=-\frac{a_ik_i}{c}t+c_0$$
with $c_0\in \mathbb{R}$,
hence,
$$\Big|\frac{dh}{dt}\Big|=e^{c_0-\frac{a_ik_i}{c}t}.$$

If $k_i\neq 0$, we get
$$f_i(t,x_i)=e^{\frac{a_ix_i+b_i}{2}}\sqrt{\Big|\pm \frac{c}{a_ik_i}e^{c_0-\frac{a_ik_i}{c}t}+c'_0\Big|}$$
with $c'_0\in \mathbb{R}$ such that $c'_0\pm \frac{\displaystyle c}{\displaystyle a_ik_i}e^{c_0-\frac{a_ik_i}{c}t}\neq 0$ for any $t\in I$.
In this case, according to Proposition \ref{poli5} (ii), $V=c\frac{\displaystyle d}{\displaystyle dt}+\sum_{i=2}^nk_i\frac{\displaystyle d}{\displaystyle dx_i}$ is a $2$-Killing vector field on $I\!\times_{{f_2}}\!\!I_2\times \cdots\times_{f_n}\!\!I_n$.

If $k_i=0$, then
$$f_i(t,x_i)=e^{\frac{a_ix_i+b_i}{2}}\sqrt{|\pm e^{c_0}t+c'_0|}$$
with $c'_0\in \mathbb{R}$ such that $c'_0\pm e^{c_0}t\neq 0$ for any $t\in I$.
In this case, according to Proposition \ref{poli5} (ii), $V=c\frac{\displaystyle d}{\displaystyle dt}$ is a $2$-Killing vector field on $I\!\times_{{f_2}}\!\!I_2\times \cdots\times_{f_n}\!\!I_n$.
\end{remark}

\begin{example}
Let $I\!\times_{{f_2}}\!\!I_2\times \cdots\times_{f_n}\!\!I_n$ be equipped with the metric $-dt^2+\sum_{i=2}^nf_i^2dx_i^2$, where $f_i:I\times I_i\rightarrow (0,\infty)$ are defined by $$f_i(t,x_i)=\sqrt{e^{t+x_i}}.$$
Then, according to Remark \ref{poli3},
$$V=-\frac{\displaystyle d}{\displaystyle dt}+\sum_{i=2}^n\frac{\displaystyle d}{\displaystyle dx_i}$$
is a $2$-Killing vector field on $I\!\times_{{f_2}}\!\!I_2\times \cdots\times_{f_n}\!\!I_n$.
\end{example}

\begin{example}
Let $I\!\times_{{f_2}}\!\!I_2\times \cdots\times_{f_n}\!\!I_n$ (such that $0\notin I$) be equipped with the metric $-dt^2+\sum_{i=2}^nf_i^2dx_i^2$, where $f_i:I\times I_i\rightarrow (0,\infty)$ are defined by $$f_i(t,x_i)=\sqrt{|t|e^{x_i}}.$$
Then, according to Remark \ref{poli3},
$$V=c\frac{\displaystyle d}{\displaystyle dt}$$
with $c\in \mathbb{R}^*$ is a $2$-Killing vector field on $I\!\times_{{f_2}}\!\!I_2\times \cdots\times_{f_n}\!\!I_n$.
\end{example}

\begin{proposition}\label{poli4}
Let $V=\sum_{i=2}^nk_i\frac{\displaystyle d}{\displaystyle dx_i}$ with $k_i\in \mathbb{R}$, $i=\overline{2,n}$, be a smooth nonzero vector field on the multiply twisted product spacetime $I\!\times_{{f_2}}\!\!I_2\times \cdots\times_{f_n}\!\!I_n$ with the metric $-dt^2+\sum_{i=2}^nf_i^2dx_i^2$. Then:

(i) $V$ is a Killing vector field if and only if
$$k_i\frac{\displaystyle \partial f_i^2}{\displaystyle \partial x_i}=0, \ \ (\forall) \ \ i=\overline{2,n};$$
in particular, if $k_i\neq 0$ for any $i=\overline{2,n}$, then $V$ is a Killing vector field if and only if the manifold is a multiply warped product spacetime;

(ii) $V$ is a $2$-Killing vector field.
\end{proposition}
\begin{proof}
The two assertions follow from Corollary \ref{cor1}.
\end{proof}

\begin{example}
Let $I\!\times_{{f_2}}\!\!I_2\times \cdots\times_{f_n}\!\!I_n$ be equipped with the metric $-dt^2+\sum_{i=2}^nf_i^2dx_i^2$. Then, according to Proposition \ref{poli4} (ii), for any twisted functions
$f_i:I\times I_i\rightarrow (0,\infty)$, the vector field
$$V=\sum_{i=2}^nk_i\frac{\displaystyle d}{\displaystyle dx_i}$$
with $k_i\in \mathbb{R}$, $i=\overline{2,n}$, is a $2$-Killing vector field on $I\!\times_{{f_2}}\!\!I_2\times \cdots\times_{f_n}\!\!I_n$.
\end{example}

\begin{remark}
Let $V=c\frac{\displaystyle d}{\displaystyle dt}+\sum_{i=2}^nk_i\frac{\displaystyle d}{\displaystyle dx_i}$ with $c,k_i\in \mathbb{R}^*$, $i=\overline{2,n}$, be a smooth vector field on the multiply twisted product spacetime $I\!\times_{{f_2}}\!\!I_2\times \cdots\times_{f_n}\!\!I_n$ with the metric $-dt^2+\sum_{i=2}^nf_i^2dx_i^2$. If
$f_i:I\times I_i\rightarrow (0,\infty)$, $f_i(t,x_i)=c_ie^{\frac{a_i(cx_i-k_it)}{c}}$ with $a_i\in \mathbb{R}^*$, $c_i\in \mathbb{R}^*_+$, then, according to Proposition \ref{poli5} (i), $V$ is a Killing vector field on $I\!\times_{{f_2}}\!\!I_2\times \cdots\times_{f_n}\!\!I_n$.
\end{remark}

\section*{Conclusions}

The paper completes some very recent studies in this direction by characterizing the $2$-Killing vector fields on one of the most general manifolds of warped product-type defined so far, namely, on multiply twisted products. By considering some particular cases, it constitutes a source of concrete examples of multiply twisted product spacetimes endowed with a $2$-Killing vector field, that could be of interest not only to differential geometers and equationists, but to theoretical physicists, as well.



\end{document}